\documentclass[12pt]{amsart}

\usepackage[left=1in,right=1in,top=1in,bottom=1in]{geometry} 

\usepackage{amsmath}

\usepackage{amssymb,amsthm,amsmath}

\usepackage{ifthen}

\usepackage{graphicx}

\usepackage{enumerate}

\usepackage{graphicx}

\usepackage{subfig}

\DeclareGraphicsExtensions{.eps}

\newtheorem{theorem}{Theorem}[section]

\newtheorem{lemma}[theorem]{Lemma}

\theoremstyle{definition}

\newtheorem{definition}[theorem]{Definition}

\newtheorem{example}[theorem]{Example}

\numberwithin{equation}{section}

\newcommand{\C}{\mathbb{C}}

\newcommand{\e}{\epsilon}

\begin{document}

\baselineskip=21pt
\title[pseudospectrum and condition pseudospectrum]
 {Perturbation analysis for the linear operator equation using pseudospectrum and condition pseudospectrum}

%   Information for first author
\author{Krishna Kumar. G}
% %    Address of record for the research reported here
\address{Department of Mathematics, Baby John Memorial Government College Chavara, Kollam, Kerala, India, 691583.}
 {\email{krishna.math@gmail.com}
%    \thanks will become a 1st page https://accounts.google.com/ServiceLogin/signinchooser?service=mail&passive=true&rm=false&continue=https%3A%2F%2Fmail.google.com%2Fmail%2F&ss=1&scc=1&ltmpl=default&ltmplcache=2&emr=1&osid=1&flowName=GlifWebSignIn&flowEntry=ServiceLoginfootnote.
%\thanks{The author wish to thank Prof. S. H. Lui, University of Manitoba, Winnipeg and Prof. S. H. Kulkarni, Indian Institute of Technology Madras for the useful discussions.}
%    Information for second author
% \author{S. H. Lui}
% \address{Department of Mathematics, University of Manitoba, Winnipeg, Manitoba, Canada, R3T 2N2.}
%  {\email{luish@cc.umanitoba.ca}
%\address{Department of Mathematics, Indian Institute of Technology, Chennai-600 036, India.}
%\email{venku@iitm.ac.in}

%\thanks{Support information for the second author.}

%    General info

%\date{January 1, 2001 and, in revised form, June 22, 2001.}

%\dedicatory{This paper is dedicated to our advisors.}

\begin{abstract} 
In this article, we consider the linear operator equation $Ax-zx= y$ in a Banach space. The relative perturbation of the solution $x$ corresponding to the perturbation of $y$, the perturbation of $A$ and the perturbation of both $A, y$ are characterized from the pseudospectrum and the condition pseudospectrum of $A$. Certain examples are given to illustrate the results. A relation between the pseudospectrum and the condition pseudospectrum of an operator are established. The distance to instability and the distance to singularity of an operator are also found from the condition pseudospectrum of the operator.
\end{abstract}
\maketitle
\section{Introduction}
Throughout this article $X$ denotes a complex Banach space and $BL(X)$ is the Banach algebra of all bounded linear operators on $X$. Define $\mathcal{I}:= \{zI: z\in \C\}$, where $I$ is the identity operator in $BL(X)$. 
\begin{definition}
Let $A\in BL(X)$. The spectrum of $A$ is denoted by $\sigma(A)$ and is defined as
\[
\sigma(A):= \{z\in \C: A-zI \ \ \mbox{is not invertible}\}.
\]
\end{definition}
The spectrum of an operator in $BL(X)$ is generalized in many ways for various applications. The pseudospectrum and the condition pseudospectrum are two important generalizations of the spectrum of an operator.
\begin{definition}
 Let $A\in BL(X)$ and $\epsilon> 0$. The $\epsilon$-pseudospectrum of $A$ is denoted by $\Lambda_{\epsilon}(A)$ and is defined as
 \[
  \Lambda_{\epsilon}(A):= \sigma(A)\cup \left\{z\in \mathbb{C}: \|(A-zI)^{-1}\|\geq \epsilon^{-1}\right\}.
 \]
\end{definition}
Hence $\sigma(A)\subseteq \Lambda_\e(A)$ for each $\e> 0$ and $\displaystyle \bigcap_{\e>0}\Lambda_\e(A)= \sigma(A)$. For more properties and various applications of the pseudospectrum one may refer to \cite{tre}. 
\begin{definition}
 Let $A\in BL(X)$ and $0<\epsilon<1$. The $\epsilon$-condition pseudospectrum of $A$ is denoted by $\sigma_{\epsilon}(A)$ and is defined as
 \[
  \sigma_{\epsilon}(A):= \sigma(A)\cup \left\{z\in \mathbb{C}: \|A-zI\|\|(A-zI)^{-1}\|\geq \epsilon^{-1}\right\}.
 \]
\end{definition}
Hence $\sigma(A)\subseteq \sigma_\e(A)$ for each $0<\e< 1$, $\displaystyle \bigcap_{0<\e<1}\sigma_\e(A)= \sigma(A)$ and $\sigma_1(A)= \C$. The condition pseudospectrum was introduced in \cite{kul} and in the same article it is called condition spectrum. Compare to several other generalizations of the spectrum, the condition pseudospectrum is proved to be algebraically close to the spectrum. Hence the condition pseudospectrum is useful to study the perturbation analysis of operators of $BL(X)$. For more properties of the condition pseudospectrum one may refer to \cite{amm, kri, kri1, lui, kul, suk}.

Let $A\in BL(X)$. Consider the linear operator equation $Ax-zx= y$ where $x,y\in X$ and $z\in \C$. The purpose of the article is to characterize the perturbation of the solution $x$ corresponding to the following perturbations from the pseudospectrum and the condition pseudospectrum of $A$.
\begin{enumerate}
 \item the perturbation of $y$
 \item the perturbation of $A$
 \item the perturbation of both $A$ and $y$
\end{enumerate}
The following is an outline of the article. In section 2, we consider certain linear operator equations $Ax-zx= y$ to illustrate that the perturbation of the solution $x$ corresponding to the perturbation of $y$, perturbation of $A$ and perturbation of both $A, y$ depends on the pseudospectrum and the condition pseudospectrum of the operator $A$. In section 3, an upper bound and lower bound is found for the relative perturbation of the solution $x$ corresponding to the relative perturbation of $y$ from the condition pseudospectrum of $A$ (Theorem \ref{perturbation1}). An upper bound is found for the relative perturbation of the solution $x$ corresponding to the relative perturbation of the operator $A$ from the pseudospectrum of $A$ (Theorem \ref{perturbation2}) and the condition pseudospectrum of $A$ (Theorem \ref{perturbation3}). An upper bound is found for the relative perturbation of the solution $x$ corresponding to the relative perturbation of both $A$ and $y$ from the condition pseudospectrum of $A$ (
Theorem \ref{perturbation4}). Certain examples are also given to illustrate the results. In section 4, a relation connecting the pseudospectrum and the condition pseudospectrum of an operator in $BL(X)$ is given as set inclusions. The distance to instability of an operator in $BL(X)$ is also characterized from the condition pseudospectrum of the operator. In section 5, the distance to singularity of an invertible operator in $BL(X)$ is characterized from the condition pseudospectrum of the operator.
\section{Preliminaries}
Let $A\in BL(X)$. Consider the the linear operator equation
\[
Ax-zx= y 
\]
where $x, y \in X$ and $z\in \C$. The following example finds the perturbation of the solution $x$ corresponding to various perturbations of $y$. We claim that the perturbation of the solution depends on the quantity $\|A-zI\|\|(A-zI)^{-1}\|$ and hence the perturbation of the solution may characterized from the condition pseudospectrum of the operator $A$.
\begin{example} 
Consider the linear systems defined by
\[
 Ax_1-x_1= y \ \ \textnormal{and} \ \ Bx_2-x_2= y
\]
where $A= \begin{bmatrix}
         1.1&0\\0&2
        \end{bmatrix}
$, $B= \begin{bmatrix}
           1.1&10\\
           0&2
          \end{bmatrix}
$ and $y= \begin{bmatrix}
           1\\1
          \end{bmatrix}
$.
Let $\delta y$ be a small perturbation on $y$ and $\delta x_1, \delta x_2$ be the corresponding perturbations of the solution $x_1, x_2$ respectively. Then
\[
(A-I)(x_1+\delta x_1)= y+ \delta y \ \ \textnormal
{and} \ \ (B-I)(x_2+ \delta x_2)= y+ \delta y.
\]
The following table gives the perturbations of the solutions for various values of $\delta y$.
\begin{center}
\begin{tabular}{|c|c|c|c|c|}
\hline $\delta y$& $\delta x_1$& $\frac{\|\delta x_1\|}{\|x_1\|}$& $\delta x_2$&$\frac{\|\delta x_2\|}{\|x_2\|}$\\
\hline \hline $[0.01, 0.01]^T$&$[0.1,0.01]^T$&0.01&$[-0.9,0.01]^T$&0.01\\
\hline $[0.01, 0.02]^T$&$[0.1,0.02]^T$&0.0101474&$[-1.9,0.02]^T$&0.0211109\\
\hline $[0.02, 0.01]^T$&$[0.2,0.01]^T$&0.0199256&$[-0.8,0.01]^T$&0.008889\\
\hline $[0.02, 0.02]^T$&$[0.2,0.02]^T$&0.02&$[-1.8,0.02]^T$&0.02\\
\hline $[0.03, -0.03]^T$&$[0.3,-0.03]^T$&0.03&$[3.3,-0.03]^T$&0.0366659\\
\hline $[0.04, -0.04]^T$&$[0.4,-0.04]^T$&0.04&$[4.4,-0.04]^T$&0.0488879\\
\hline $[0.05, -0.05]^T$&$[0.5,-0.05]^T$&0.05&$[5.5,-0.05]^T$&0.0611098\\
\hline $[0.06, -0.06]^T$&$[0.6,-0.06]^T$&0.06&$[6.6,-0.06]^T$&0.0733318\\
\hline
\end{tabular}
\end{center}
\end{example}
\begin{center}
\begin{figure}[h]
\begin{tabular}{cc}
 \includegraphics[width=85mm]{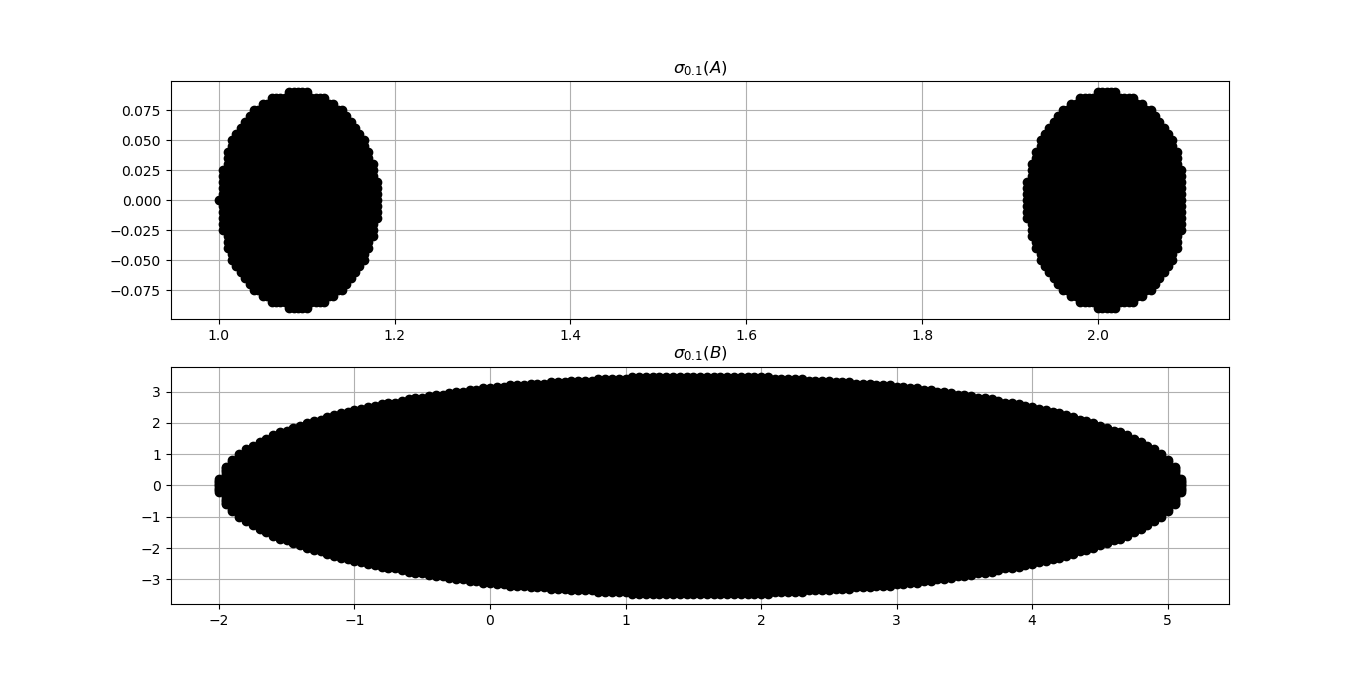}& \includegraphics[width=85mm]{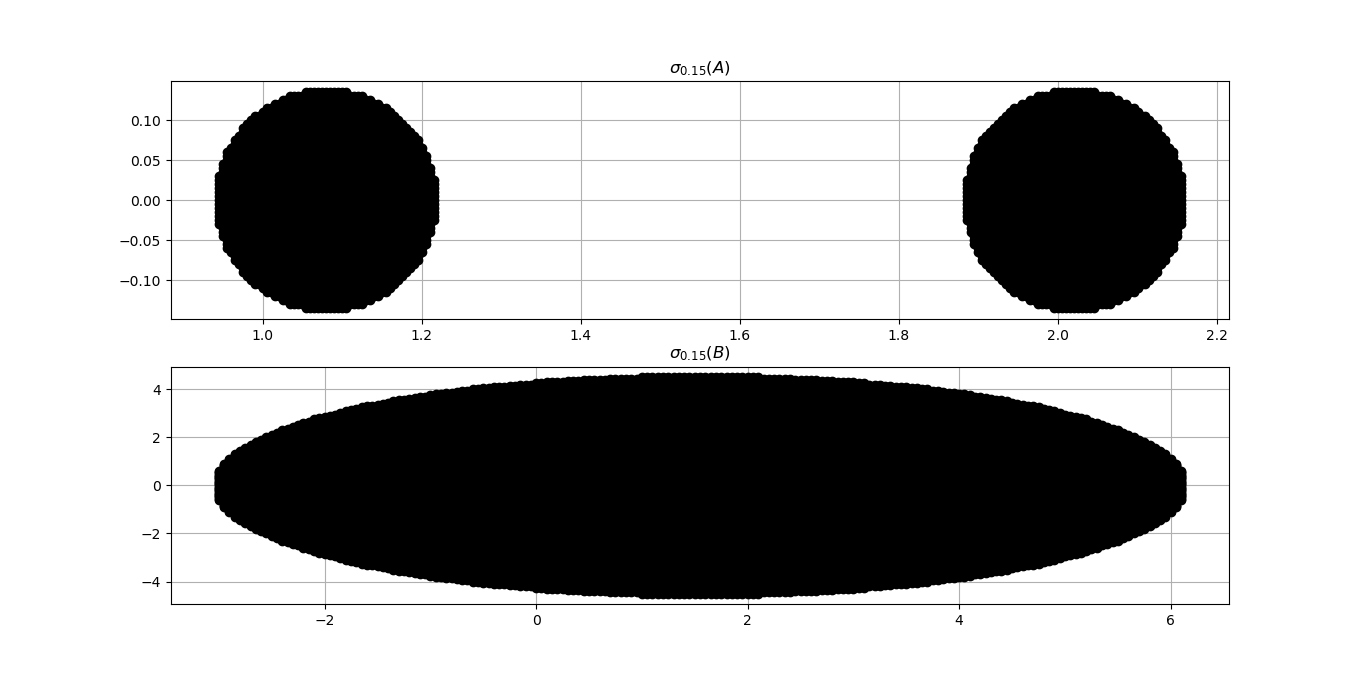}\\
 \includegraphics[width=85mm]{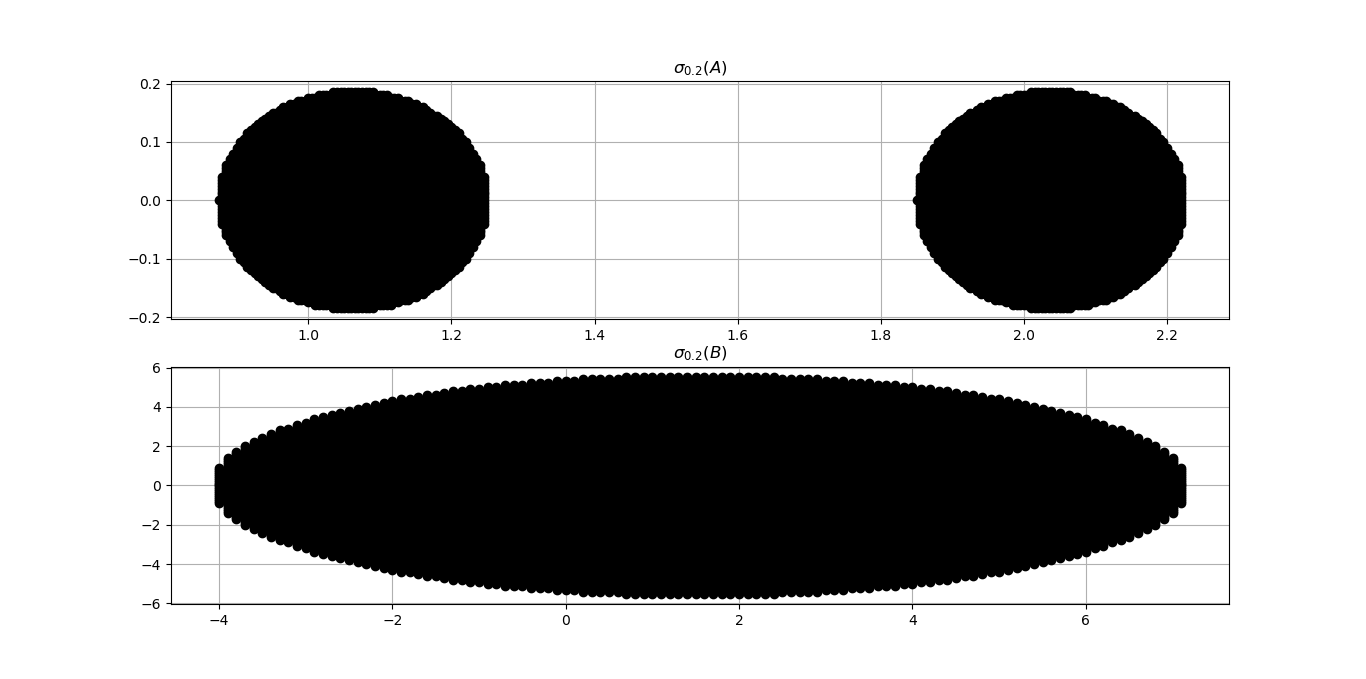}& \includegraphics[width=85mm]{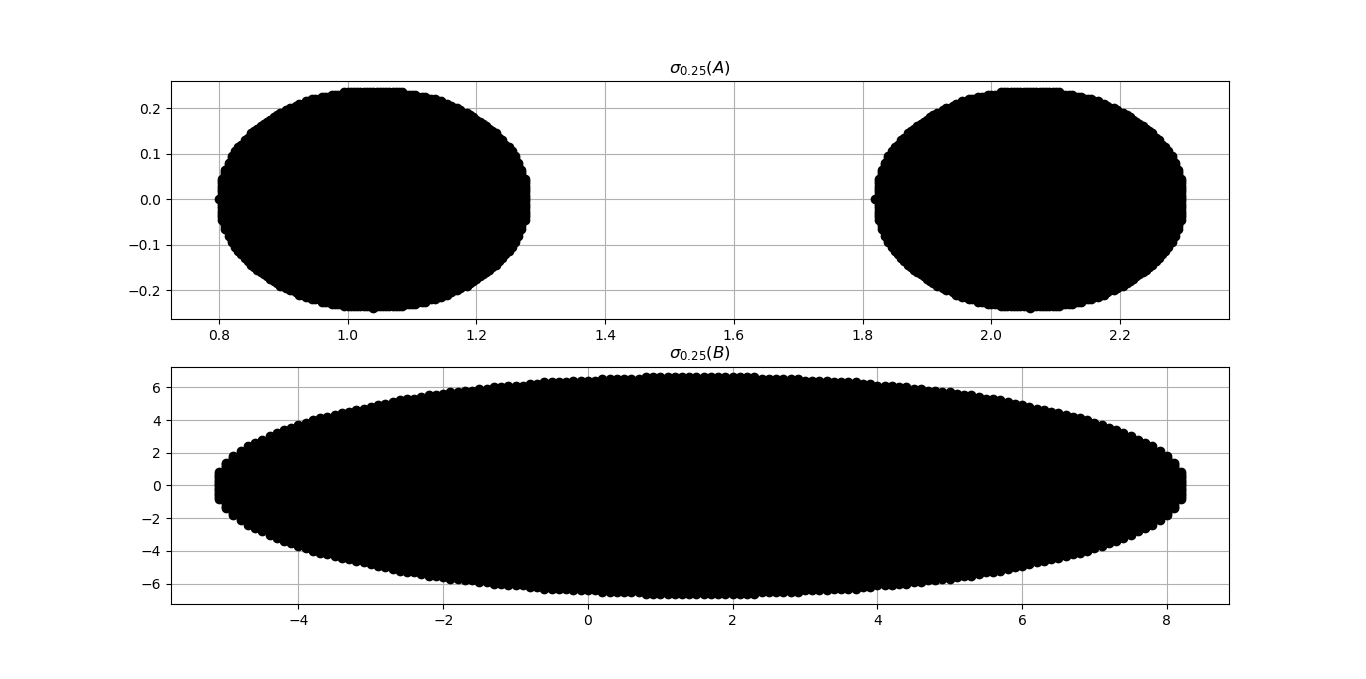}
\end{tabular}
\caption{}
\end{figure}
\end{center}
In Figure 1, the $\e$-condition pseudospectrum of $A$ and $B$ are found for $\e= 0.1, 0.15, 0.2$ and $0.25$.  For $z\in \C$ and $A\in BL(X)$ define 
\[
\kappa(z,A):= \sup\{0<\e<1: z\notin \sigma_\e(A)\}.
\] 
Then $\kappa(z,B)\leq \kappa(z,A)$ for every $z\in \C$ and $\sigma_\e(A)\subseteq \sigma_\e(B)$ for each $0< \e< 1$. The following example finds the perturbation of the solution $x$ corresponding to various perturbations of $A$. We claim that the perturbation of the solution depends on the quantity $\|(A-zI)^{-1}\|$ and hence the perturbation of the solution may characterized from the pseudospectrum of the operator $A$.
\begin{example}
Consider the linear systems defined by
\[
 Ax_1-2x_1= y \ \ \textnormal{and} \ \ Bx_2-2x_2= y
\]
where $A= \begin{bmatrix}
         1&0&-1\\0&2.1&1\\0&0&3
        \end{bmatrix}
$, $B= \begin{bmatrix}
           1&10&10\\
           0&2.1&10\\
           0&0&3
          \end{bmatrix}
$ and $y= \begin{bmatrix}
           1\\1\\1
          \end{bmatrix}
$. Let $\Delta A, \Delta B$ be small perturbations on $A, B$ and $\delta x_1, \delta x_2$ be the corresponding perturbations of the solution $x_1, x_2$ respectively. Then
\[
(A-2I+\Delta A)(x_1+\delta x_1)= y \ \ \textnormal{and} \ \ (B-2I+\Delta B)(x_2+ \delta x_2)= y.
\]
The following table gives the perturbation of the solution $x$ corresponding to various perturbations of $A$ and $B$. Choose $\Delta A= \Delta B= 
\begin{bmatrix}
\e_1&0&0\\
0&\e_2&0\\
0&0&0
\end{bmatrix}
$.
\begin{center}
\begin{tabular}{|c|c|c|c|c|}
\hline $\Delta A= \Delta B$& $\delta x_1$& $\frac{\|\delta x_1\|}{\|x_1\|}$& $\delta x_2$&$\frac{\|\delta x_2\|}{\|x_2\|}$\\
\hline \hline $\e_1= 0.01, \e_2= 0.01$&$[-0.02020202,0,0]^T$&0.0090346&$[73.64463,8.18182,0]^T$&0.0827413 \\
\hline $\e_1= 0.02, \e_2= 0.02$&$[-0.04081633,0,0]^T$&0.0182536&$[134.87755,15,0]^T$&0.1515397\\
\hline $\e_1= 0.03, \e_2= 0.03$&$[-0.06185567,0,0]^T$&0.0276626&$[186.55908,20.76923,0]^T$&0.2096085\\
\hline $\e_1= 0.04, \e_2= 0.04$&$[-0.08333333,0,0]^T$&0.0372677&$[230.73214,25.71428,0]^T$&0.2592425\\
\hline $\e_1= 0.05, \e_2= 0.05$&$[-0.10526316,0,0]^T$&0.0470751&$[268.89473,30,0]^T$&0.302125\\
\hline $\e_1= 0.06, \e_2= 0.06$&$[-0.12765957,0,0]^T$&0.0570910&$[302.170212,33.75,0]^T$&0.339517\\
\hline $\e_1= 0.07, \e_2= 0.07$&$[-0.15053763,0,0]^T$&0.0673224&$[ 331.41745, 37.05882,0]^T$&0.3723843\\
\hline
\end{tabular}
\end{center}
\end{example}
\begin{center}
\begin{figure}[h]
\begin{tabular}{cc}
 \includegraphics[width=85mm]{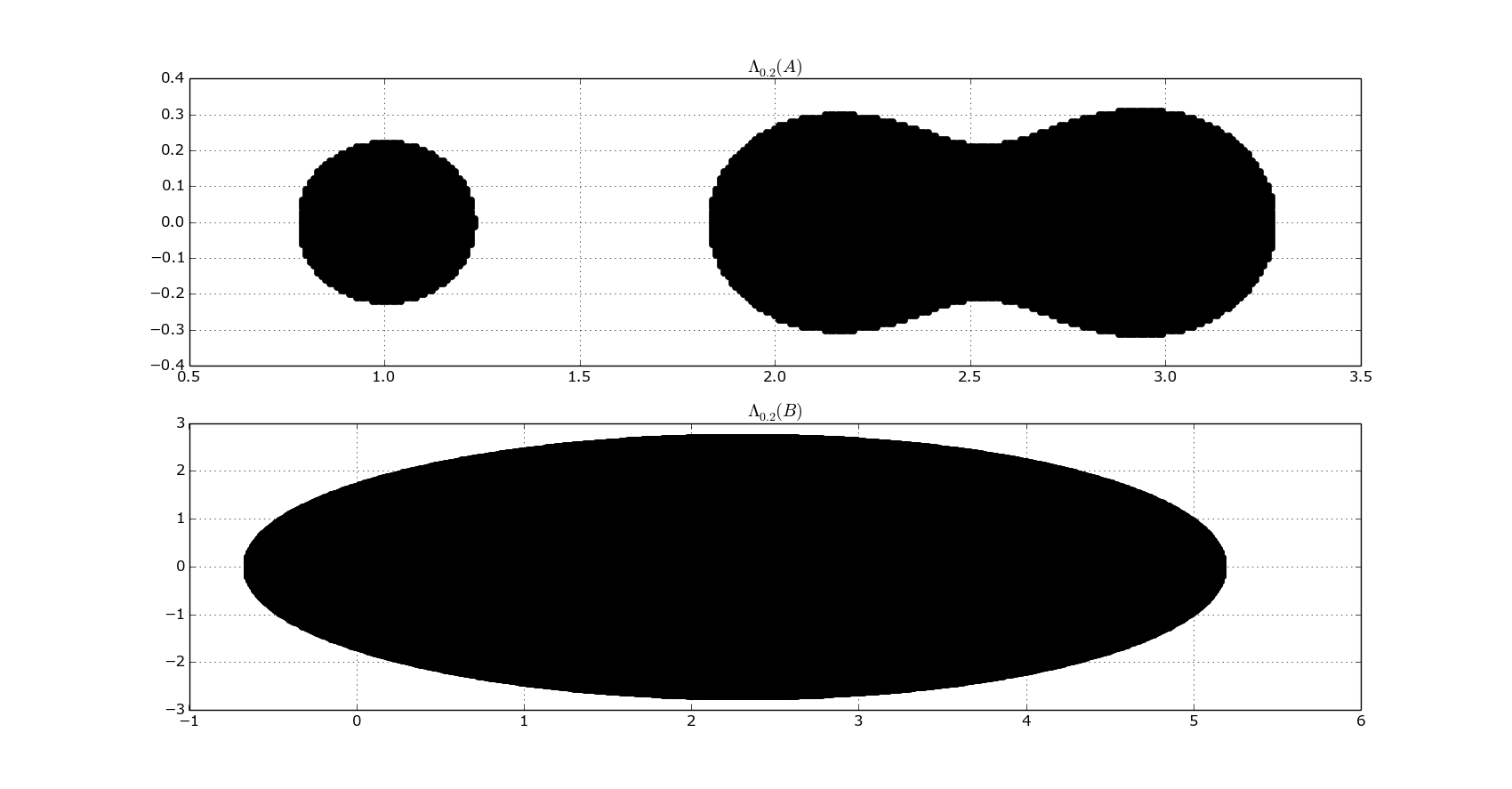}& \includegraphics[width=85mm]{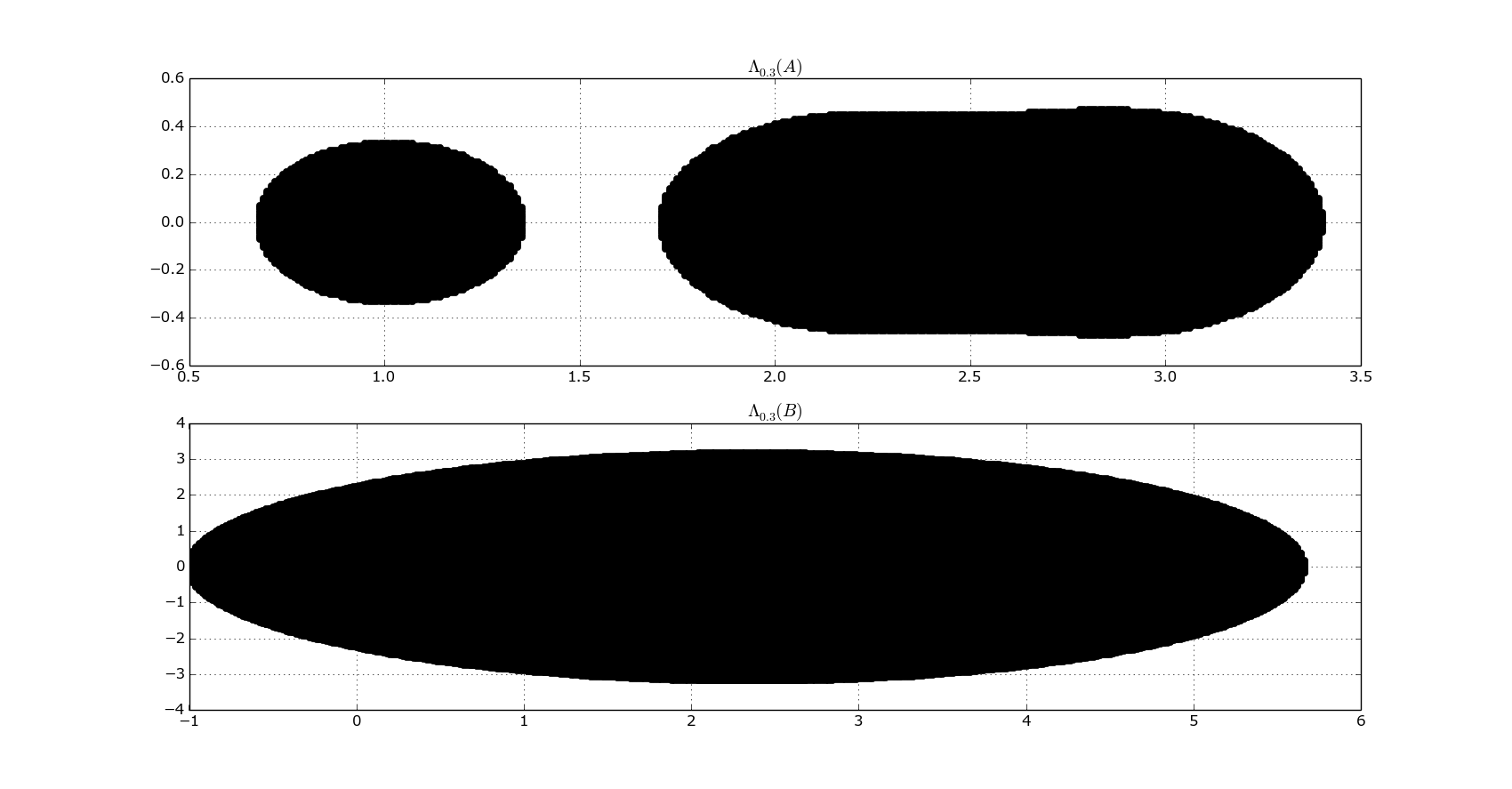}\\
 \includegraphics[width=85mm]{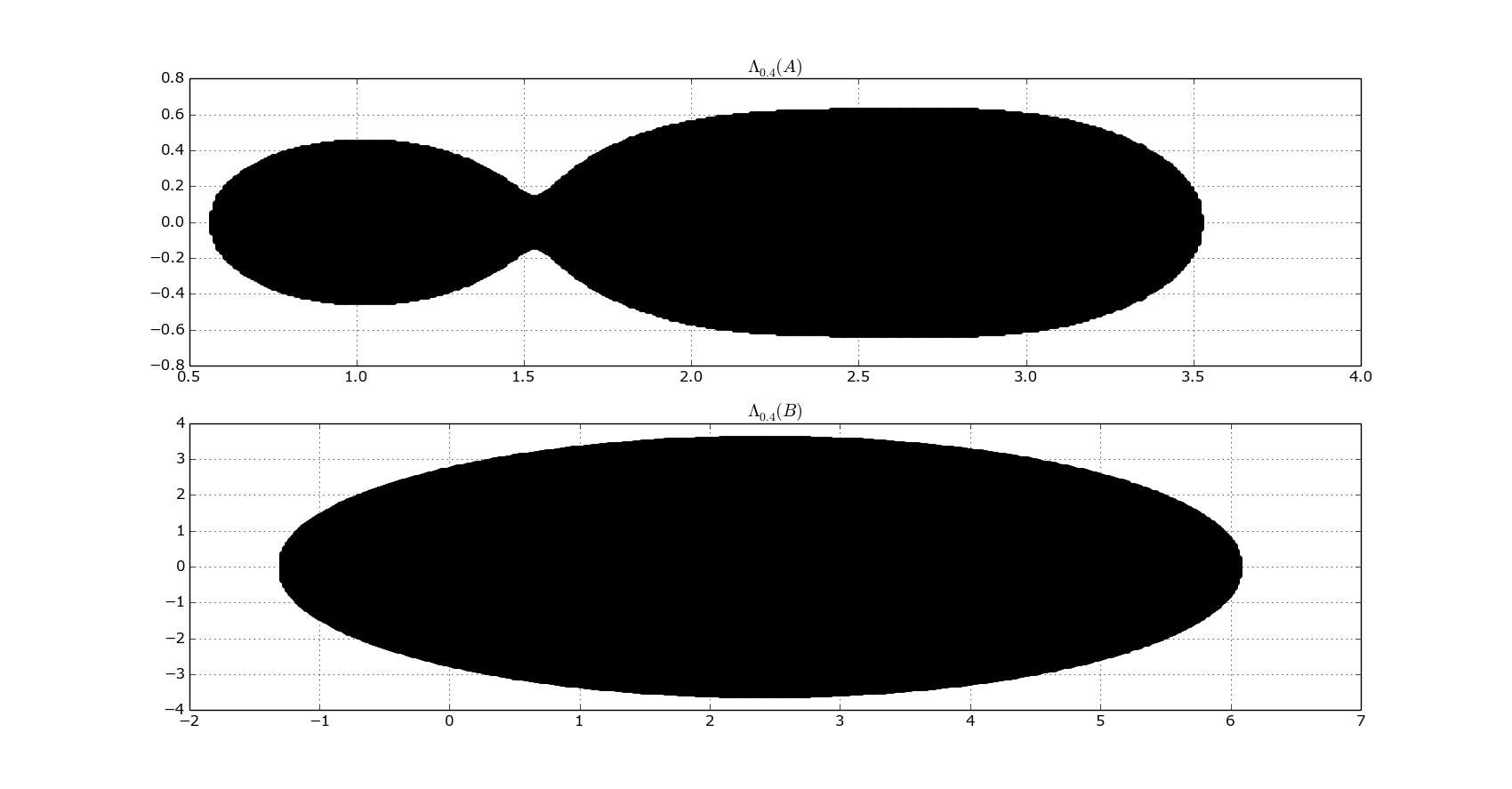}& \includegraphics[width=85mm]{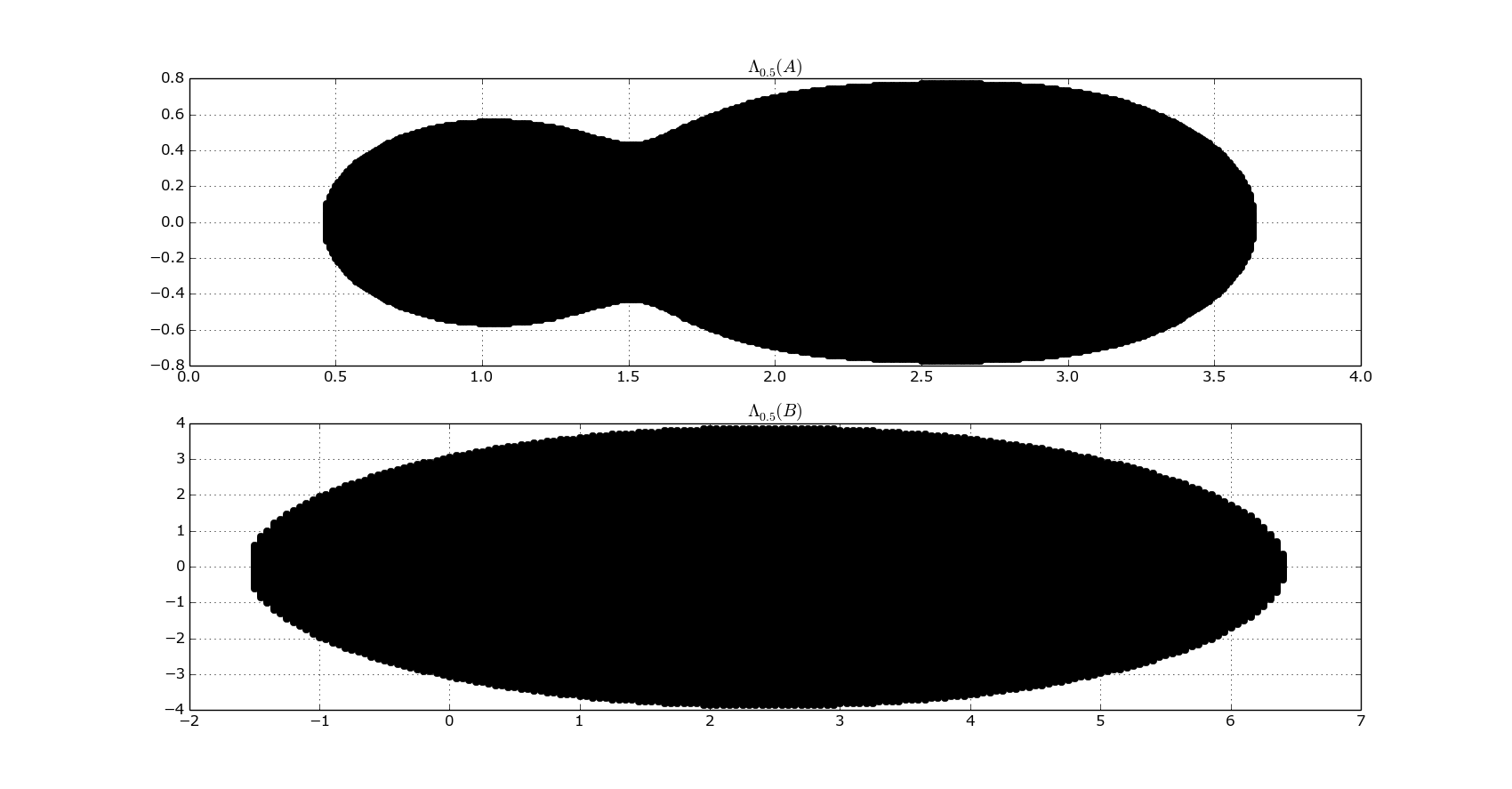}
\end{tabular}
\caption{}
\end{figure}
\end{center}
In Figure 2, the $\e$-pseudospectrum of $A$ and $B$ are found for $\e= 0.2, 0.3, 0.4$ and $0.5$. For $z\in \C$ and $A\in BL(X)$ define 
\[
\kappa_1(z,A):= \sup\{\e>0: z\notin \Lambda_\e(A)\}.
\] 
Then $\kappa_1(z,B)\leq \kappa_1(z,A)$ for every $z\in \C$ and $\Lambda_\e(A)\subseteq \Lambda_\e(B)$ for each $\e>0$. The following example finds the perturbation of the solution $x$ corresponding to various perturbations of $A, y$. We claim that the perturbation of the solution depends on the quantity $\|A-zI\|\|(A-zI)^{-1}\|$ and hence the perturbation of the solution may characterized from the condition pseudospectrum of the operator $A$.
\begin{example}
Consider the linear systems defined by
\[
 Ax_1-3x_1= y \ \ \textnormal{and} \ \ Bx_2-3x_2= y
\]
where $A= \begin{bmatrix}
         1&0&-1\\0&2&1\\0&0&3.1
        \end{bmatrix}
$, $B= \begin{bmatrix}
           1&10&10\\
           0&2&10\\
           0&0&3.1
          \end{bmatrix}
$ and $y= \begin{bmatrix}
           1\\1\\1
          \end{bmatrix}
$. Let $\Delta A, \Delta B$ be a small perturbations on $A, B$ respectively
and $\delta y$ be a small perturbation on $y$. Further let $\delta x_1, \delta x_2$ be the corresponding perturbations of the solution $x_1, x_2$. Then
\[
(A-3I+\Delta A)(x_1+\delta x_1)= y+ \delta y \ \ \textnormal{and} \ \ (B-3I+\Delta B)(x_2+ \delta x_2)= y+\delta y.
\]
The following table gives the perturbation of the solution $x$ corresponding to various perturbations of $A, B$ and $\delta y$. Choose $\Delta A= \Delta B= 
\begin{bmatrix}
0&0&0\\
0&\e_1&0\\
0&0&\e_2
\end{bmatrix}
$.
\begin{center}
\begin{tabular}{|c|c|c|c|}
\hline $\Delta A= \Delta B$&$\delta y$& $\delta x_1$& $\delta x_2$\\
\hline 
\hline $\e_1= 0.01, \e_2= 0.01$&$[0.01,0.01,0.01]^T$&$[0.404,-0.745,-0.818]^T$&$[-40.468,-7.274, -0.818]^T$ \\
\hline $\e_1= 0.02, \e_2= 0.02$&$[0.02,0.02,0.02]^T$&$[0.74,-1.367,-1.5]^T$&$[-74.040,-13.306,-1.5]^T$\\
\hline $\e_1= 0.03, \e_2= 0.03$&$[0.03,0.03,0.03]^T$&$[1.023,-1.893,-2.076]^T$&$[-102.302,-18.380,-2.076]^T$\\
\hline $\e_1= 0.04, \e_2= 0.04$&$[0.04,0.04,0.04]^T$&$[1.265,-2.345,-2.571]^T$&$[-126.389,-22.702,-2.571]^T$\\
\hline $\e_1= 0.05, \e_2= 0.05$&$[0.05,0.05,0.05]^T$&$[1.475,-2.736,-3]^T$&$[-147.130,-26.421,-3]^T$\\
\hline $\e_1= 0.06, \e_2= 0.06$&$[0.06,0.06,0.06]^T$&$[1.657,-3.079,-3.375]^T$&$[-165.149,-29.648,-3.375]^T$\\
\hline $\e_1= 0.07, \e_2= 0.07$&$[0.07,0.07,0.07]^T$&$[1.817,-3.382,-3.705]^T$&$[-180.923,-32.471,-3.705]^T$\\
\hline
\end{tabular}
\end{center}
\end{example}
\begin{center}
\begin{figure}[h]
\begin{tabular}{cc}
 \includegraphics[width=85mm]{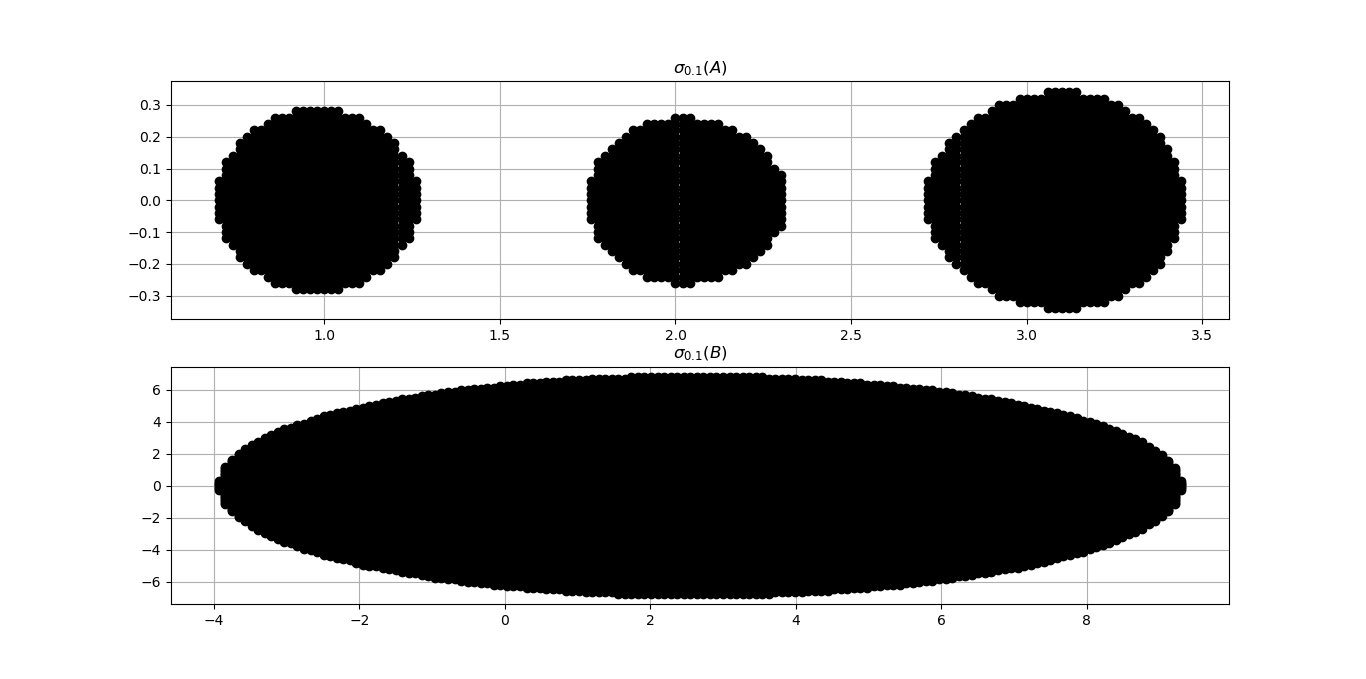}& \includegraphics[width=85mm] {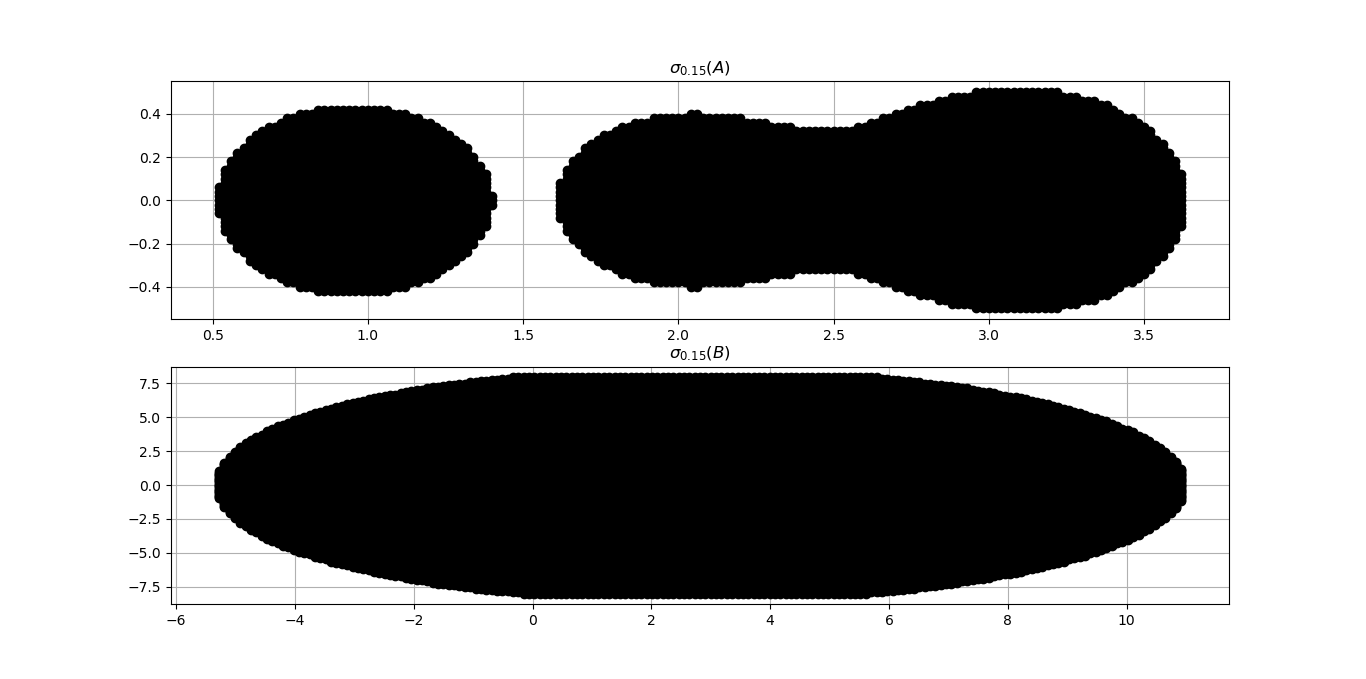}\\
 \includegraphics[width=85mm]{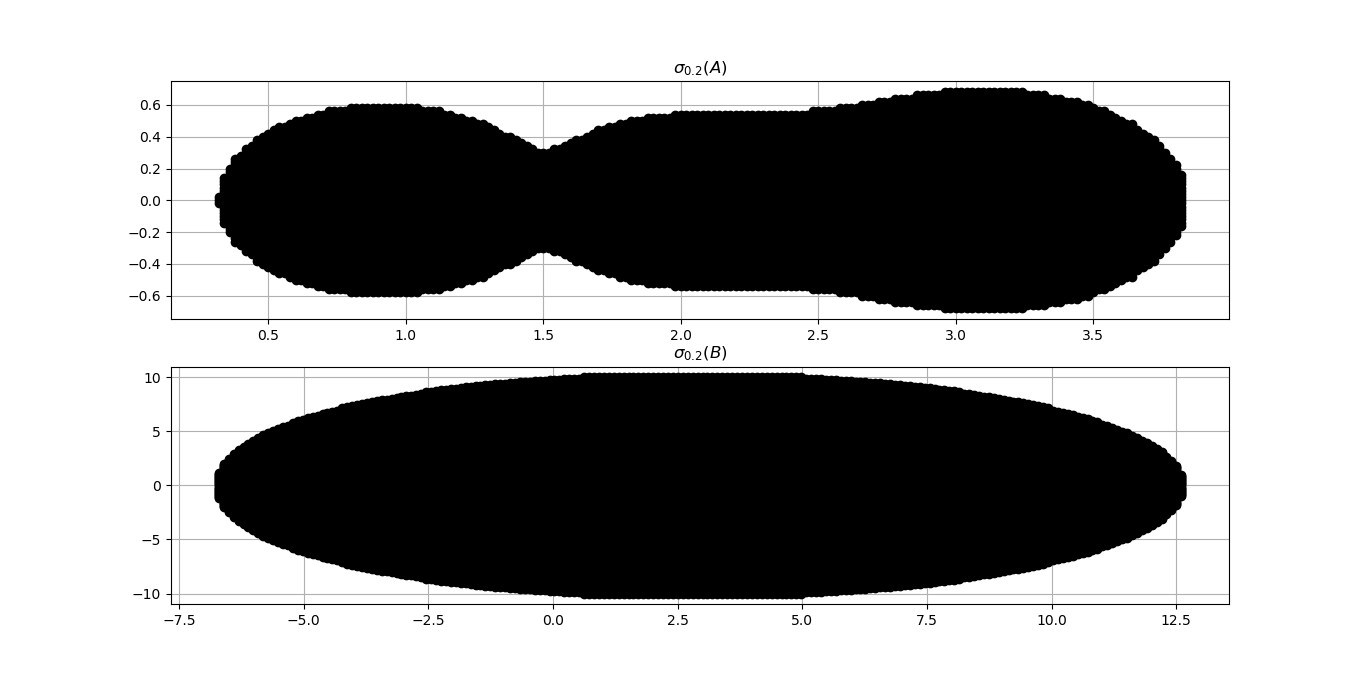}& \includegraphics[width=85mm] {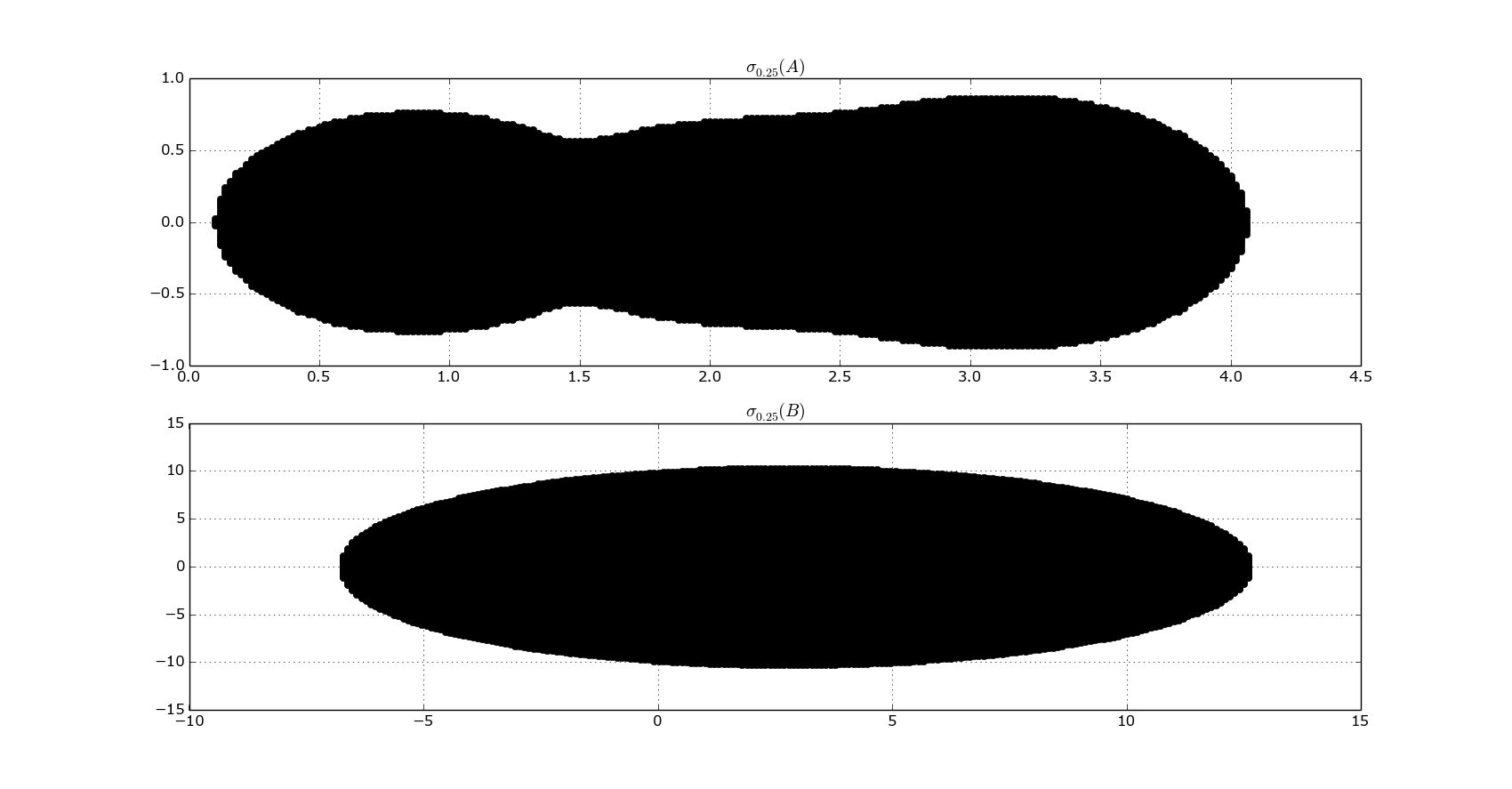}
\end{tabular}
\caption{}
\end{figure}
\end{center}
In Figure 3, the $\e$-condition pseudospectrum of $A$ and $B$ are found for $\e= 0.1, 0.15, 0.2$ and $0.25$. Then $\kappa(z,B)\leq \kappa(z,A)$ for every $z\in \C$ and $\sigma_\e(A)\subseteq \sigma_\e(B)$ for each $0< \e< 1$.
\section{The stability of linear operator equations}
\subsection{Perturbation of $y$} 
Let $A\in BL(X)$ and $Ax-zx= y$ where $x, y\in X$ and $z\in \C$. The following theorem finds an upper bound and lower bound for the relative perturbations of the solution $x$ from the condition pseudospectrum of $A$.
\begin{theorem}\label{perturbation1}
Let $A\in BL(X)$ and $Ax-zx= y$ where $x, y\in X$ and $z\notin \sigma(A)$. Further let $\delta x$ be the perturbation of the solution $x$ corresponding to a perturbation $\delta y$ of $y$. Then  
\[
\kappa(z,A) \, \frac{\|\delta y\|}{\|y\|}\leq \frac{\|\delta x\|}{\|x\|} \leq \frac{1}{\kappa(z,A)}\frac{\|\delta y\|}{\|y\|}
\]
where $\displaystyle \kappa(z,A):= \sup\{0<\e<1: z\notin \sigma_\e(A)\}.$
\end{theorem}
\begin{proof}
We have
 \[
  (A-zI)(x+ \delta x)= y+ \delta y.
 \]
Since $z\notin \sigma(A)$,
\begin{eqnarray}\label{eqn1}
\|x\|\leq \|(A-zI)^{-1}\| \|y\| \ \ \mbox{and} \ \ \|y\|\leq \|A-zI\|\|x\|.
\end{eqnarray}
Also
\begin{eqnarray}\label{eqn2}
\|\delta y\|\leq \|A-zI\| \|\delta x\| \ \ \mbox{and} \ \ \|\delta x\|\leq \|(A-zI)^{-1}\|\|\delta y\|.
\end{eqnarray}
From (\ref{eqn1}) and (\ref{eqn2}) it follows that
\[
 \frac{1}{\|(A-zI)\| \|(A-zI)^{-1}\|} \frac{\|\delta y\|}{\|y\|}\leq \frac{\|\delta x\|}{\|x\|} \leq \|A-zI\|\|(A-zI)^{-1}\|\frac{\|\delta y\|}{\|y\|}.
\]
If $z\notin \sigma_\e(A)$ for some $0<\e<1$, then
\[
 \e \, \frac{\|\delta y\|}{\|y\|}< \frac{\|\delta x\|}{\|x\|}< \frac{1}{\e}\frac{\|\delta y\|}{\|y\|}.
\]
Define $\displaystyle \kappa(z,A):= \sup\{0<\e<1: z\notin \sigma_\e(A)\}.$ Then 
\[
\kappa(z,A) \, \frac{\|\delta y\|}{\|y\|}\leq \frac{\|\delta x\|}{\|x\|} \leq \frac{1}{\kappa(z,A)}\frac{\|\delta y\|}{\|y\|}.
\]
\end{proof}
\begin{example}
Consider the operator equation $Ax-zx= y$ defined by
\[
A= \begin{bmatrix}
0&1&&&\\
\frac{1}{4}&0&1&&\\
&\ddots&\ddots&\ddots&\\
&&\frac{1}{4}&0&1\\
&&&\frac{1}{4}&0\\
\end{bmatrix}_{10\times 10}  \mbox{and} \ \ \ \ y= \begin{bmatrix}.
1\\1\\\vdots\\1\\1
\end{bmatrix}_{10\times 1}.
\]
The following table gives the relative perturbation of the solution $x$  for various values of $z$ and for $\delta y= \begin{bmatrix}0.1&0&0&0&0&0&0&0&0&0\end{bmatrix}^{T}$. The table also gives the upper bound and lower bound for the relative perturbation of the solution found from the condition pseudospectrum of $A$. 
\vspace{0.3 cm}
\begin{center}
\begin{tabular}{|c|c|c|c|c|}
\hline$z$&$\frac{\|\delta x\|}{\|x\|}$&$\kappa(z,A)$&$\kappa(z,A) \frac{\|\delta y\|}{\|y\|}$&$\frac{1}{\kappa(z,A)}\frac{\|\delta y\|}{\|y\|}$\\
\hline \hline $2$&0.01451478&0.25562528&0.00808358&0.12370754\\
\hline $2+i$&0.01916041&0.34480349&0.01090364&0.09171245\\
\hline $3-i$&0.02157139&0.22394260&0.00708168&0.14120929\\
\hline $1-i$&0.02192375&0.22394260&0.00708168&0.14120929\\
\hline $2-i$&0.01916041&0.34480349&0.01090364&0.09171245\\
\hline $3+3i$&0.02634769&0.62518073&0.01976995&0.05058181\\
\hline $4i$&0.03238426&0.69642891&0.02202301&0.04540704\\
\hline 
\end{tabular}
\end{center}
\end{example}
\subsection{Perturbation of $A$}
Let $A\in BL(X)$ and $Ax-zx= y$ where $x, y\in X$ and $z\in \C$. The following theorem finds an upper bound for $\frac{\|\delta x\|}{\|x\|}, \frac{\|\delta x\|}{\|x+\delta x\|}$ from the pseudospectrum of $A+\Delta A$, $A$ respectively. 
\begin{theorem}\label{perturbation2}
Let $A\in BL(X)$ and $Ax-zx= y$ where $x, y\in X$ and $z\notin \sigma(A)$. Further let $\delta x$ be the perturbation of the solution $x$ corresponding to the perturbation $\Delta A$ of $A$. Then 
\begin{enumerate}
 \item $\frac{\|\delta x\|}{\|x\|}\leq \frac{\|\Delta A\|}{\kappa_1(z,A+\Delta A)}$.
 \item $\frac{\|\delta x\|}{\|x+\delta x\|}\leq \frac{\|\Delta A\|}{\kappa_1(z,A)}$.
\end{enumerate}
where $\displaystyle \kappa_1(z,A):= \sup\{\e>0: z\notin \Lambda_\e(A)\}.$
\end{theorem}
\begin{proof}
\begin{enumerate}
\item We have $(A-zI+ \Delta A)(x+ \delta x)= y.$ Then
\[
 (A-zI)\delta x+ \Delta A(x+ \delta x)= 0.
\]
If $z\notin \Lambda_\e(A+\Delta A)$ for some $\e>0$. Then $\|\delta x\|\leq \|(A+\Delta A-zI)^{-1}\|\|\Delta A\|\|x\|$ and $\frac{\|\delta x\|}{\|x\|}< \frac{\|\Delta A\|}{\e}.$  Define $\displaystyle \kappa_1(z,A+\Delta A):= \sup\{\e>0: z\notin \Lambda_\e(A+\Delta A)\}.$ Then 
\[
\frac{\|\delta x\|}{\|x\|}\leq \frac{\|\Delta A\|}{\kappa_1(z,A+\Delta A)}.
\]
\item If $z\notin \Lambda_\e(A)$ for some $\e>0$. Then $\|\delta x\|\leq \|(A-zI)^{-1}\|\|\Delta A\|\|x+\delta x\|$ and $\frac{\|\delta x\|}{\|x+\delta x\|}< \frac{\|\Delta A\|}{\e}.$  Define $\displaystyle \kappa_1(z,A):= \sup\{\e>0: z\notin \Lambda_\e(A)\}.$ Then
\[
\frac{\|\delta x\|}{\|x+\delta x\|}\leq \frac{\|\Delta A\|}{\kappa_1(z,A)}.
\]
\end{enumerate}
\end{proof}
\begin{lemma}\label{lem4}
Let $A\in BL(X)$. Then for every  $z\in \C$
\[
 \sup\{\|A-\lambda I\|: \lambda\in \sigma(A)\}\leq \|A-zI\|.
\]
\end{lemma}
\begin{proof}
If $\lambda\in \sigma(A)$ and $z\in \C$. Then $\lambda- z\in \sigma(A-zI)$ and $|\lambda- z|\leq \|A-zI\|$. Also
\[
 \|A-\lambda I\|= \|A-zI+zI-\lambda I\|\leq \|A-zI\|+ |\lambda-z|\leq 2\|A-zI\|.
\]
\end{proof}

The following theorem finds an upper bound for $\frac{\|\delta x\|}{\|x\|}, \frac{\|\delta x\|}{\|x+\delta x\|}$ from the condition pseudospectrum of $A+\Delta A, A$ respectively. 
\begin{theorem}\label{perturbation3}
Let $A\in BL(X)$ and $Ax-zx= y$ where $x, y\in X$ and $z\notin \sigma(A)$. Further let $\delta x$ be a perturbation of the solution $x$ corresponding to the perturbation $\Delta A$ of $A$.  Define $\displaystyle \kappa(z,A):= \sup\{0<\e<1: z\notin \sigma_\e(A)\}.$ Then 
\begin{enumerate}
 \item $\frac{\|\delta x\|}{\|x\|}\leq \frac{2\|\Delta A\|}{M_1 \, \kappa(z,A+\Delta A)}$ where $M_1:= \sup\{\|A+\Delta A- \lambda I\|: \lambda\in \sigma(A+\Delta A)\}$.
 \item $\frac{\|\delta x\|}{\|x+\delta x\|}\leq \frac{2\|\Delta A\|}{M_2 \, \kappa(z,A)}$ where $M_2:= \sup\{\|A- \lambda I\|: \lambda\in \sigma(A)\}$.
\end{enumerate}
\end{theorem}
\begin{proof}
\begin{enumerate}
\item We have $(A-zI+ \Delta A)(x+ \delta x)= y$ and
\[
(A-zI)\delta x+ \Delta A(x+ \delta x)= 0.
\]
If $z\notin \sigma_\e(A+\Delta A)$ for some $0<\e<1$. Then
\[\|\delta x\|\leq \frac{\|A+\Delta A-zI\|\|(A+\Delta A-zI)^{-1}\|\|\Delta A\|\|x\|}{\|A+\Delta A-zI\|}.\]
From Lemma \ref{lem4},
\[
 \frac{\|\delta x\|}{\|x\|}\leq \frac{2\|\Delta A\|}{M_1 \, \e}
\]
where $M_1= \sup\{\|A+\Delta A- \lambda I\|: \lambda\in \sigma(A+\Delta A)\}$. Hence 
\[
\frac{\|\delta x\|}{\|x\|}\leq \frac{2\|\Delta A\|}{M_1 \, \kappa(z,A+\Delta A)}
\] 
where $\kappa(z,A+\Delta A)$ is defined above.
\item If $z\notin \sigma_\e(A)$ for some $0<\e<1$. Then
\[
\|\delta x\|\leq \frac{\|A-zI\|\|(A-zI)^{-1}\|\|\Delta A\|\|x+\delta x\|}{\|A-zI\|}.
\]
From Lemma \ref{lem4},
\[
 \frac{\|\delta x\|}{\|x+\delta x\|}\leq \frac{2\|\Delta A\|}{M_2 \ \e}
\]
where $M_2= \max\{\|A- \lambda I\|: \lambda\in \sigma(A)\}$. Hence  
\[
\frac{\|\delta x\|}{\|x+\delta x\|}\leq \frac{2\|\Delta A\|}{M_2 \, \kappa(z,A)}
\]
where $\kappa(z,A)$ is defined above.
\end{enumerate}
\end{proof}
\begin{example}
Consider the operator equation $Ax-zx= y$ defined by $A= \begin{bmatrix}1&-6&7&-9\\1&-5&0&0\\0&1&-5&0\\0&0&1&-5\end{bmatrix}$ and $y= \begin{bmatrix}1&1&1&1\end{bmatrix}^T$. Let $\Delta A= \begin{bmatrix}
 -0.01&0&0&0\\0&-0.01&0&0\\0&0&0&0\\0&0&0&0\end{bmatrix}$ be a small perturbation on $A$. Then
 \[
 \sigma(A)= \{0.00964896, -3.72221248, -5.14371824+1.17699479\, i, -5.14371824-1.17699479\; i \},
 \]
 \[
  \sigma(A+\Delta A)= \{-2.40617352\times e^{-5}, -3.726616, -5.14667997+1.17985088\, i, -5.14667997-1.17985088\, i\}.
 \]
The following table gives the estimates of $\frac{\|\delta x\|}{\|x+\delta x\|}$ for various values of $z$. The table also gives the upper bound for $\frac{\|\delta x\|}{\|x+\delta x\|}$ found from the pseudospectrum and condition pseudospectrum of $A$.
\begin{center}
\begin{tabular}{|c|c|c|c|}
\hline $z$&$\frac{\|\delta x\|}{\|x+\delta x\|}$&$ \frac{\|\Delta A\|}{\kappa_1(z,A)}$&$ \frac{2\|\Delta A\|}{M_2 \, \kappa(z,A)}$\\
\hline\hline $0$&1.00250311&3.20374633&6.28729589\\
\hline $0.001+0.001i$&1.11144480&3.54979202&6.96656727\\
\hline $0.1$&0.11116547&0.33536750&0.65971972\\
\hline $-3.5$&0.01890969&0.10268623&0.19736876\\
\hline $-5+i$&0.02005952&0.06879138&0.13682253\\
\hline $1+i$&0.00711180&0.01821627&0.03689838\\
\hline $10+10i$&0.00039271&0.00096510&0.00324741\\
\hline  
\end{tabular}
\end{center}
The following table gives the estimates of the relative perturbation of the solution $x$ for various values of $z$. The table also gives the upper bound for $\frac{\|\delta x\|}{\|x\|}$ found from the pseudospectrum and condition pseudospectrum of $A$.
\begin{center} 
\begin{tabular}{|c|c|c|c|}
\hline $z$&$\frac{\|\delta x\|}{\|x\|}$&$\frac{\|\Delta A\|}{\kappa_1(z,A+\Delta A)}$&$\frac{2\|\Delta A\|}{M_1 \ \kappa(z,A+\Delta A)}$\\
\hline\hline $0$&400.49840503&1283.50359817&2519.32007747\\
\hline $0.001+0.001i$&6.73531833&21.57230484&42.34410702\\
\hline $0.1$&0.10005389&0.30265747&0.59548750\\
\hline $-3.5$&0.01856115&0.10081974&0.19377512\\
\hline $-5+i$&0.01989405&0.06762120&0.13448096\\
\hline $1+i$&0.00707485&0.01811715&0.03670612\\
\hline $10+10i$&0.00039262&0.00096463&0.00324679\\
\hline 
\end{tabular}
\end{center}
\end{example}
\subsection{Perturbation of both $A$ and $y$}
Let $A\in BL(X)$ and $Ax-zx= y$ where $x, y\in X$ and $z\in \C$. The following theorem finds an upper bound for the relative perturbations of the solution $x$ corresponding to a perturbation of both $A$ and $y$ from the condition pseudospectrum of $A$. 
\begin{theorem}\label{perturbation4}
Let $Ax-zx= y$ where $A\in BL(X)$, $x, y\in X$ and $z\notin \sigma(A)$. Further let $\delta x$ be the perturbation of the solution $x$ corresponding to the perturbation $\Delta A, \delta y$ on $A, y$ respectively. Then
\[
  \frac{\|\delta x\|}{\|x\|}\leq \frac{1}{\kappa(z,A)[1-\|(A-zI)^{-1}\Delta A\|]}\left(\frac{\|\delta y\|}{\|y\|}+ \frac{\|\Delta A\|}{\|A-zI\|}\right)
\]
where $\displaystyle \kappa(z,A):= \sup\{0<\e<1: z\notin \sigma_\e(A)\}.$
\end{theorem}
\begin{proof}
We have 
\[
(A-zI+ \Delta A)(x+\delta x)= y+ \delta y.
\]
From Theorem 7.12 of \cite{book1},
\[
  \frac{\|\delta x\|}{\|x\|}\leq \frac{\|A-zI\|\|(A-zI)^{-1}\|}{1-\|(A-zI)^{-1}\Delta A\|}\left(\frac{\|\delta y\|}{\|y\|}+ \frac{\|\Delta A\|}{\|A-zI\|}\right). 
\]
Suppose $z\notin \sigma_\e(A)$ for some $0<\e<1$. Then
\[
  \frac{\|\delta x\|}{\|x\|}< \frac{1}{\e[1-\|(A-zI)^{-1}\Delta A\|]}\left(\frac{\|\delta y\|}{\|y\|}+ \frac{\|\Delta A\|}{\|A-zI\|}\right).
\]
Hence 
\[
  \frac{\|\delta x\|}{\|x\|}\leq \frac{1}{\kappa(z,A)[1-\|(A-zI)^{-1}\Delta A\|]}\left(\frac{\|\delta y\|}{\|y\|}+ \frac{\|\Delta A\|}{\|A-zI\|}\right)
\]
where $\displaystyle \kappa(z,A):= \sup\{0<\e<1: z\notin \sigma_\e(A)\}.$
\end{proof}
The following table gives the estimates of the relative perturbation of the solution $x$ for various values of $z$. The table also gives the upper bound for $\frac{\|\delta x\|}{\|x\|}$ found from the condition pseudospectrum of $A$.
\begin{example}
Consider the operator equation $Ax-zx= y$ defined by $A= \begin{bmatrix}0&1&2&3&4&5\\1&0&1&2&3&4\\2&1&0&1&2&3\\3&2&1&0&1&2\\ 4&3&2&1&0&1\\ 5&4&3&2&1&0\end{bmatrix}$ and $y= \begin{bmatrix}1\\2\\3\\4\\5\\6\end{bmatrix}^T$. Let $\Delta A= \begin{bmatrix}
 -0.01&0&0&0&0&0\\0&-0.01&0&0&0&0\\0&0&0&0&0&0\\0&0&0&0&0&0\\0&0&0&0&0&0\\0&0&0&0&0&0\\\end{bmatrix}$, $\delta y= \begin{bmatrix}0.01\\0.02\\0.03\\0.04\\0\\0\end{bmatrix}$ be small perturbations on $A$, $y$ respectively. Then
\begin{center}
 \begin{tabular}{|c|c|c|}
 \hline  $z$& $\frac{\|\delta x\|}{\|x\|}$& $ \frac{1}{\kappa(z,A)[1-\|(A-zI)^{-1}\Delta A\|]}\left(\frac{\|\delta y\|}{\|y\|}+ \frac{\|\Delta A\|}{\|A-zI\|}\right)$\\
 \hline \hline 0&0.0268467778225&0.150444148229\\ 
 \hline 0.2&0.0222940217983&0.107632701309\\
 \hline 0.4&0.0193775211883&0.0832300065481\\
 \hline 0.6&0.0173217152417&0.0674667943276\\
 \hline -0.1&0.0302644475105&0.186762288147\\
 \hline -0.3&0.0422038073974&0.353665883561\\
 \hline -0.5&0.0572957745312&2.62816178085\\
 \hline 
 \end{tabular}
\end{center}
\end{example}
\section{Characterization of the distance to instability of an operator from the condition pseudospectrum}
\begin{definition}
Let $A\in BL(X)$. Then $A$ is said to be stable if 
\[
\sigma(A)\cap \{z: \mbox{Re}\,z> 0\}= \emptyset.
\]
\end{definition}
\begin{definition}
Let $A\in BL(X)$. The distance to instability of $A$ is denoted by $d_1(A)$ and is defined as
\begin{eqnarray*}
d_1(A)&:=& \min\{\|E\|: A+E \ \mbox{is not stable}\},\\
&=& \min\{\|E\|: \sigma(A+E)\cap \{z: \mbox{Re}\, z >0\}\neq \emptyset\}.
\end{eqnarray*}
\end{definition}
The following lemmas find the relation connecting the pseudospectrum and the condition pseudospectrum of an operation $A\in BL(X)$.
\begin{lemma}\label{lem2}
Let $A\in BL(X)$ and $\e> 0$. Then $\displaystyle \Lambda_\e(A)\subseteq \sigma_{\frac{2\e}{M(A)}}(A)$ where $M(A):= \sup\{\|A-\lambda I\|: \lambda\in \sigma(A)\}$.
\end{lemma}
\begin{proof}
Suppose $z\in \Lambda_\e(A)$ for some $\e>0$. Then $\displaystyle \|(A-zI)^{-1}\|\geq \frac{1}{\e}$ and 
\[
\|A-zI\| \|(A-zI)^{-1}\|\geq \frac{\|A-zI\|}{\e}.
\]
From Lemma \ref{lem4}, 
\[
\|A-zI\| \|(A-zI)^{-1}\|\geq \frac{M(A)}{2\e}
\]
where $M(A)$ is defined above. Hence $\displaystyle z\in \sigma_{\frac{2\e}{M(A)}}(A)$.
\end{proof}
\begin{lemma}\label{lem3}
Let $A\in BL(X)$ and $\e> 0$. Then $\sigma_{\frac{\e}{\e+2\|A\|}}(A)\subseteq \Lambda_\e(A).$
\end{lemma}
\begin{proof}
Suppose $\displaystyle z\in \sigma_{\frac{\e}{\e+2\|A\|}}(A)$ for some $\e>0$. Then
\[
\|A-zI\| \|(A-zI)^{-1}\|\geq \frac{\e+2\|A\|}{\e}.
\]
\[
\|(A-zI)^{-1}\|\geq \frac{\e+2\|A\|}{\e\|A-zI\|}\geq \frac{\e+2\|A\|}{\e(|z|+\|A\|)}.
\]
Since $\displaystyle |z|\leq \frac{1+\frac{\e}{\e+2\|A\|}}{1- \frac{\e}{\e+2\|A\|}}\|A\|$ (Theorem 2.9 of \cite{kul}), we have 
\[
\|(A-zI)^{-1}\|\geq \frac{\e+2\|A\|}{\e\left(\frac{1+\frac{\e}{\e+2\|A\|}}{1- \frac{\e}{\e+2\|A\|}}\|A\|+\|A\|\right)}= \frac{1}{\e}.
\]
\end{proof}
The following theorem finds an upper bound and lower bound to the distance to singularity of an operator $A\in BL(X)$ from the condition pseudospectrum of $A$.
\begin{theorem}
 Let $A\in BL(X)$. Define $M(A):= \sup\{\|A-\lambda I\|: \lambda\in \sigma(A)\}$, then
 \[
 d_1(A)\geq \max\left\{\e>0: \sigma_{\frac{2\e}{M(A)}}(A)\cap \{z: \textnormal{Re}\, z>0\}= \emptyset \right\},
 \]
\[
d_1(A)\leq \max\left\{\e>0: \sigma_{\frac{\e}{\e+2\|A\|}}(A) \cap \{z: \textnormal{Re}\, z>0\}= \emptyset\right\}.
\]
\end{theorem}
\begin{proof}
From the definition of the pseudospectrum and $d_1(A)$ we have 
\begin{eqnarray*}
d_1(A)&=& \min\{\e>0: \Lambda_\e(A)\cap \{z: \mbox{Re}\, z>0\} \neq \emptyset\}\\
&=& \max\{\e>0: \Lambda_\e(A)\cap \{z: \mbox{Re}\, z>0\}= \emptyset\}.
\end{eqnarray*}
From Lemma \ref{lem2} and Lemma \ref{lem3}, for $\e>0$,
\[
\sigma_{\frac{\e}{\e+2\|A\|}}(A)\subseteq \Lambda_\e(A)\subseteq \sigma_{\frac{2\e}{M(A)}}(A).
\]
Hence
\[
\sigma_{\frac{2\e}{M(A)}}(A)\cap \{z: \mbox{Re}\, z>0\}= \emptyset \ \ \ \Longrightarrow \ \ \
\Lambda_\e(A) \cap \{z: \mbox{Re}\, z>0\}= \emptyset,
\]
and
\[
\Lambda_\e(A) \cap \{z: \mbox{Re}\, z>0\}= \emptyset \ \ \ \Longrightarrow \ \ \ \sigma_{\frac{\e}{\e+2\|A\|}}(A) \cap \{z: \mbox{Re}\, z>0\}= \emptyset.
\]
\end{proof}
\section{Characterization of the distance to singularity of an operator from the condition pseudospectrum}
\begin{definition}
Let $A\in BL(X)$. The distance to singularity of $A$ is denoted by $d_2(A)$ and is defined as
\[
d_2(A):= \min\{\|E\|: A+E \ \mbox{is singular}\}.
\]
\end{definition}
The following is an equivalent definition of the pseudospectrum of an operator in $BL(X)$ \cite{tre}.
\begin{definition}\label{eqn3}
Let $A\in BL(X)$ and $\e>0$. Then
\[
\Lambda_\e(A)= \{z: z\in \sigma(A+E), \|E\|\leq \e\}.
\]
\end{definition}
\begin{lemma}\label{lem1}
Let $A\in BL(X)$ and $\e>0$. Then $0\notin \Lambda_\e(A)$ if and only if $0\notin \sigma_{\frac{\e}{\|A\|}}(A)$.
\end{lemma}
\begin{proof}
Suppose $0\notin \Lambda_\e(A)$ for some $\e> 0$. Then $0\notin \sigma(A)$ and $\|A^{-1}\|<\frac{1}{\e}$. Hence $\|A\|\neq 0$ and $\|A\|\|A^{-1}\|<\frac{\|A\|}{\e}$. Thus $0\notin \sigma_{\frac{\e}{\|A\|}}(A)$. The proof of the reverse inclusion is quite similar.
\end{proof}
The following theorem characterizes the distance to singularity of an operator $A\in BL(X)$ from the condition pseudospectrum of $A$.
\begin{theorem}
Let $A\in BL(X)\smallsetminus \mathcal{I}$. Then
\[
d_2(A)= \max\left\{\e> 0: 0\notin \sigma_\frac{\e}{\|A\|}(A)\right\}. 
\] 
\end{theorem}
\begin{proof}
From the definition of $d_2(A)$ we have
\[
d_2(A)= \min\{\|E\|: 0\in \sigma(A+E)\}.
\]
From Definition \ref{eqn3}, 
\begin{eqnarray*}
d_2(A)&=& \min\{\e>0: 0\in \Lambda_\e(A)\},\\
&=& \max\{\e>0: 0\notin \Lambda_\e(A)\}.
\end{eqnarray*}
From Lemma \ref{lem1},
\begin{eqnarray*}
d_2(A)&=& \max\left\{\e> 0: 0\notin \sigma_\frac{\e}{\|A\|}(A)\right\}.
\end{eqnarray*}
\end{proof}
\bibliographystyle{amsplain}

\end{document}